\documentclass[12pt,twoside]{amsart}
\usepackage{geometry}
\usepackage{amssymb,amsmath,amsthm, amscd, enumerate, mathrsfs}
\usepackage{graphicx, hhline}
\usepackage[all]{xy}

\usepackage[shortlabels]{enumitem}
\setlist[enumerate]{leftmargin=56pt,labelsep=
8pt,itemsep=4pt,label=\upshape{(\theequation.\arabic*)}}
\usepackage{mathtools}
\mathtoolsset{showonlyrefs}

\title{A remark on toric foliations} 
\author{Osamu Fujino and Hiroshi Sato} 
\date{2024/2/22, version 0.21}
\address{Department of 
Mathematics, Graduate School of Science, 
Kyoto University, Kyoto 606-8502, Japan}
\email{fujino@math.kyoto-u.ac.jp}
\address{Department of Applied Mathematics, 
Faculty of Sciences, Fukuoka University, 
8-19-1, Nanakuma, Jonan-ku, Fukuoka 814-0180, Japan}
\email{hirosato@fukuoka-u.ac.jp}

\newcommand{\Star}[0]{{\operatorname{Star}}}
\newcommand{\mult}[0]{{\operatorname{mult}}}
\newtheorem{thm}{Theorem}[section]
\newtheorem*{claim}{Claim}

\theoremstyle{definition}
\newtheorem{defn}[thm]{Definition}
\newtheorem*{ack}{Acknowledgments}

\makeatletter
    \renewcommand{\theequation}{%
    \thesection.\arabic{equation}}
    \@addtoreset{equation}{section}
\makeatother

\begin{document}

\begin{abstract}
If a toric foliation on a projective 
$\mathbb Q$-factorial toric variety has an extremal ray 
whose length is longer than the rank of the foliation, 
then the associated extremal contraction is a projective 
space bundle and the foliation is the relative tangent sheaf of 
the extremal contraction. 
\end{abstract}

\maketitle 

\tableofcontents 

\section{Introduction}\label{section1}

Let us start with the definition of {\em{foliations}} on 
normal algebraic varieties. 

\begin{defn}[Foliations and toric foliations]\label{definition1.1}
A {\em{foliation}} on a normal algebraic variety 
$X$ is a nonzero saturated subsheaf 
$\mathscr F\subset \mathscr T_X$ that is closed under the Lie bracket, 
where $\mathscr T_X$ is  the tangent 
sheaf of $X$. 
We note that the {\em{rank}} of the foliation $\mathscr F$ 
means the rank of the coherent sheaf $\mathscr F$. 

We further assume that $X$ is toric. 
Then a foliation $\mathscr F$ on $X$ is called {\em{toric}} if 
the sheaf $\mathscr F$ is torus equivariant. 
\end{defn}

The following result on toric foliations is a starting point of 
this paper. 

\begin{thm}[{see \cite{pang}}]\label{theorem1.2}
Let $X=X(\Sigma)$ be a $\mathbb Q$-factorial toric variety 
with its fan $\Sigma$ in the lattice $N\simeq \mathbb Z^n$. 
Then there exists a one-to-one correspondence 
between the set of toric foliations on $X$ and 
the set of complex vector subspaces $V\subset N_{\mathbb C}:=N\otimes 
_{\mathbb Z} \mathbb C\simeq \mathbb C^n$. 

Let $\mathscr F_V$ be the toric foliation associated to 
a complex vector subspace $V\subset N_{\mathbb C}$ 
$($here, we should remark that the 
rank of $\mathscr F_V$ is $\dim_\mathbb{C} V$$)$. 
Then 
\begin{equation}
K_{\mathscr F_V}:=-c_1(\mathscr F_V)=-\sum _{\rho\subset V}D_\rho 
\end{equation} holds, that is, the first Chern class of 
$\mathscr F_V$ is $\sum _{\rho\subset V}D_\rho$, 
where $D_\rho$ is the torus invariant prime divisor corresponding 
to the one-dimensional cone $\rho$ in $\Sigma$. 
In particular, we 
have 
\begin{equation}
K_{\mathscr F_V}=K_X+\sum _{\rho\not\subset V}D_\rho. 
\end{equation}
\end{thm}

For the basics of toric foliations, see 
also \cite{chang-chen} and \cite{wang}. 
By \cite{fujino-sato}, we see that we can run the minimal model 
program with respect to $K_{\mathscr F}$ for any foliation $\mathscr F$ 
on a projective $\mathbb Q$-factorial toric variety $X$. 
For more details on the 
toric foliated minimal model program, see \cite{chang-chen} and 
\cite{wang}. In this paper, we establish: 

\begin{thm}[Main Theorem]\label{theorem1.3}
Let $X$ be a 
projective $\mathbb Q$-factorial 
toric variety and let $\mathscr F$ be a toric 
foliation of rank $r$ on $X$. 
Then 
\begin{equation}
l_{\mathscr F}(R):=\min _{[C]\in R}\{-K_{\mathscr F}\cdot C\}\leq r+1
\end{equation}
holds for every extremal ray $R$ of $\overline{\mathrm{NE}}(X)=
\mathrm{NE}(X)$. Moreover, 
if $l_{\mathscr F}(R)>r$ 
holds for some extremal ray $R$ of 
$\mathrm{NE}(X)$, then the contraction morphism 
$\varphi_R\colon X\to Y$ associated to 
$R$ is a $\mathbb P^r$-bundle over $Y$. 
In this case, $\mathscr F=\mathscr T_{X/Y}$ holds, where 
$\mathscr T_{X/Y}$ is the relative 
tangent sheaf of $\varphi_R\colon X\to Y$. 
In particular, $\mathscr F$ is locally free. 
\end{thm}

We note that we call $l_{\mathscr F}(R)$ 
the {\em{length}} of 
an extremal ray $R$ with respect to the foliation $\mathscr F$. 
We will use Reid's description of the toric extremal 
contraction morphisms in \cite{reid} (see also \cite[Chapter 14]{matsuki}) 
for the proof of Theorem \ref{theorem1.3}. 
This paper can be seen as a continuation of \cite{fujino} (see 
also \cite{fujino2}). 

\begin{ack}\label{acknowledgments}
The first author was partially 
supported by JSPS KAKENHI Grant Numbers 
JP19H01787, JP20H00111, JP21H00974, JP21H04994. 
The second author was partially
supported by JSPS KAKENHI Grant Number 
JP18K03262. 
The authors thank the referee very much for 
many useful comments and suggestions. 
\end{ack}

\section{Preliminaries on toric varieties}\label{section2}

Let $N\simeq \mathbb Z^n$ be a lattice of rank $n$. 
A toric variety $X(\Sigma)$ is associated to a {\em{fan}} $\Sigma$, 
a collection of convex cones $\sigma\subset N_\mathbb R:=
N\otimes _{\mathbb Z}\mathbb R$ satisfying: 
\begin{itemize}
\item 
Each convex cone $\sigma$ is a 
rational polyhedral cone in the sense that  
there are finitely many 
$n_1, \ldots, n_s\in N\subset N_{\mathbb R}$ such 
that 
\begin{equation}
\sigma=\{r_1n_1+\cdots +r_sn_s; \ r_i\geq 0\}=
:\langle n_1, \ldots, n_s\rangle, 
\end{equation} 
and it is strongly convex in the sense that  
\begin{equation}
\sigma \cap -\sigma=\{0\}. 
\end{equation} 
\item Each face $\tau$ of a convex cone $\sigma\in \Sigma$ 
is again an element in $\Sigma$. 
\item The intersection of two cones in $\Sigma$ is a face of 
each. 
\end{itemize}
The {\em{dimension}} $\dim \sigma$ of a cone $\sigma$ is 
the dimension of the linear space 
$\mathbb R\sigma=\sigma +(-\sigma)$ spanned 
by $\sigma$. We define the sublattice $N_{\sigma}$ of 
$N$ generated (as a subgroup) by $\sigma\cap N$ as 
follows: 
\begin{equation}
N_{\sigma}:=\sigma\cap N+(-\sigma\cap N). 
\end{equation}
If $\sigma$ is a $k$-dimensional simplicial 
cone, and $v_1,\ldots, v_k$ are the 
first lattice points along the edges of $\sigma$, 
then $\sigma=\langle v_1, \ldots, v_k\rangle$ holds. 
The {\em{multiplicity}} of $\sigma$ is defined 
to be the {\em{index}} of the lattice 
generated by the $\{v_1, \ldots, v_k\}$ in the lattice $N_{\sigma}$; 
\begin{equation}
\mult (\sigma):=[N_{\sigma}:\mathbb Zv_1+\cdots +
\mathbb Zv_k]. 
\end{equation} 
We note that the affine toric variety 
$X(\sigma)$ associated to the cone $\sigma$ is smooth if and only 
if $\mult (\sigma)=1$. 
We also note that a toric variety $X(\Sigma)$ is $\mathbb Q$-factorial 
if and only if each cone $\sigma\in \Sigma$ is simplicial 
(see e.g.~\cite[Lemma 14-1-1]{matsuki}). 

The {\em{star}} of a cone $\tau\in\Sigma$ can be defined abstractly 
as the set of cones $\sigma$ in $\Sigma$ that 
contain $\tau$ as a face. Such cones $\sigma$ are 
determined by their images in 
$N(\tau):=N/{N_{\tau}}$, that is, by 
\begin{equation}
\overline \sigma:=\left(\sigma+(N_{\tau})_{\mathbb R}\right)/ 
(N_{\tau})_{\mathbb R}\subset N(\tau)_{\mathbb R}. 
\end{equation}  
These cones $\{\overline \sigma ; \tau\prec \sigma\}$ 
form a fan in $N(\tau)$, and we denote this fan by 
$\Star(\tau)$. 
We set $V(\tau)=X(\Star (\tau))$, that is, the toric variety associated to 
the fan $\Star (\tau)$. 
It is well known that $V(\tau)$ is an $(n-k)$-dimensional 
closed toric subvariety of $X(\Sigma)$, where $\dim \tau=k$. 
If $\dim V(\tau)=1$ (resp.~$n-1$), then we call $V(\tau)$ 
a {\em{torus invariant curve}} (resp.~{\em{torus invariant 
divisor}}). 
For the details about the correspondence between $\tau$ and 
$V(\tau)$, see \cite[3.1 Orbits]{fulton}. 

\section{Proof of Theorem \ref{theorem1.3}}\label{section3} 

In this section, we will prove Theorem \ref{theorem1.3}. 

\begin{proof}[Proof of Theorem \ref{theorem1.3}]
We assume that the toric variety $X$ is associated to a fan $\Sigma$, 
which is a collection of convex cones in $N\simeq \mathbb Z^n$ as explained 
in Section \ref{section2}. 
In particular, $\dim X=n$. 
It is well known that every extremal 
ray of $\overline{\mathrm{NE}}(X)=
\mathrm{NE}(X)$ is 
spanned by a torus invariant 
curve (see e.g.~\cite[Theorem 14-1-4]{matsuki}). 
Let $R$ be an extremal ray of $\mathrm{NE}(X)$. 
If $l_{\mathscr F}(R)\leq r$ holds, then there is nothing to prove. 
Therefore, we assume 
that $-K_{\mathscr F}\cdot C>r$ holds for every torus 
invariant curve $C$ with $[C]\in R$. 
We further assume that 
$C$ corresponds to an $(n-1)$-dimensional 
cone $W=\langle v_1, \ldots, v_{n-1}\rangle\in\Sigma$, 
where $v_1, \ldots, v_{n-1}$ are primitive vectors. 
Let $v_n, v_{n+1}\in N$ be the two primitive vectors 
such that they together with $W$ generate the 
two $n$-dimensional cones 
$\sigma,\sigma'\in\Sigma$, respectively. 
As usual, we can write 
\begin{equation}\label{equation3.1}
a_1v_1+\cdots +a_{n-1}v_{n-1}+a_nv_n+a_{n+1}v_{n+1}=0
\end{equation} 
such that $a_i$ is an integer for every $i$ with 
$\gcd(a_1, \ldots, a_{n+1})=1$ 
and $a_n,a_{n+1}>0$. 
We should remark that for a $1$-dimensional cone $\langle v\rangle\in\Sigma$, 
where $v\in N$ is a primitive vector, 
we have the following formula for the intersection number 
of $D_v:=V(\langle v\rangle)$ 
with $C$ 
(see e.g.~\cite[Proposition 6.4.4]{cls}): 
\[
D_v\cdot C=
\left\{
\begin{array}{ccl}
0 & \cdots & v\not\in\{v_1,\ldots,v_{n+1}\} \\
\displaystyle{\frac{a_i\mult(W)}{a_n\mult(\sigma)}} 
& \cdots & v=v_i\mbox{ for } 1\le i\le n \\
\displaystyle{\frac{\mult(W)}{\mult(\sigma')}} & \cdots & v=v_{n+1}
\end{array}
\right.
\] 
In this setting, 
\cite[Proposition 14-1-5 (i)]{matsuki} says that 
for $1\le i\le n-1$ with $a_i>0$, we have 
\begin{equation*}
\langle\{v_1,\ldots,v_n\}\setminus\{v_i\}\rangle\in\Sigma
\end{equation*} 
and 
\begin{equation*}
[V(\langle\{v_1,\ldots,v_n\}\setminus\{v_i\}\rangle)]\in R.
\end{equation*} 
Thus, we may assume that 
\begin{equation}
a_1\leq \cdots \leq a_n\leq a_{n+1} 
\end{equation} 
by changing the order. 
In particular, the above formula tells us that 
$D\cdot C\le 1$ for any torus invariant prime divisor $D$ on $X$. 
Since we have 
\begin{equation*}
-K_{\mathscr F}\cdot C=\sum_{v_i\in V}V(\langle v_i\rangle)\cdot C>r, 
\end{equation*} 
we obtain $1\le i_1<i_2<\cdots <i_r<i_{r+1}\le n+1$ such that 
\begin{equation}
v_{i_1}, v_{i_2}, \ldots, v_{i_r}, v_{i_{r+1}}\in V. 
\end{equation}
Since the rank of $\mathscr F$ is $r$, we obtain 
$\dim _{\mathbb C}V=r$ and 
\begin{equation}
\begin{split}
V&=\mathbb R\langle v_{i_1}, v_{i_2}, \ldots, v_{i_r}, v_{i_{r+1}}\rangle 
\otimes _{\mathbb R}\mathbb C\\
&=\mathbb R\langle v_{i_1}, v_{i_2},\ldots, v_{i_r}\rangle 
\otimes _{\mathbb R}\mathbb C. 
\end{split}
\end{equation} 
In particular, we have $v_i\not \in V$ for every 
$i\not\in \{i_1, i_2,\ldots, i_r, i_{r+1}\}$. 
Then $a_i=0$ holds in \eqref{equation3.1} for 
every $i\not \in \{i_1, i_2,\ldots, i_r, i_{r+1}\}$. 
Thus, $\{i_1, i_2,\ldots, i_r, i_{r+1}\}=
\{n-r+1, n-r+2, \ldots, n, n+1\}$ holds and \eqref{equation3.1} becomes 
\begin{equation}\label{equation3.2}
a_{n-r+1}v_{n-r+1}+\cdots +a_{n+1}v_{n+1}=0. 
\end{equation} 
We define $n$-dimensional cones 
\begin{equation}
\sigma_i:=\langle v_1,\ldots,v_{i-1},v_{i+1},\ldots,v_{n+1}\rangle
\in \Sigma
\end{equation}  
for $n-r+1\leq i\leq n+1$. 
We put $\mu_{i,j}=\sigma_i\cap \sigma_j
\in \Sigma$ for $i\ne j$. 
We note that 
\begin{equation}\label{equation3.3}
\begin{split}
r<-K_{\mathscr F}\cdot V(\mu_{k, n+1})&\leq 
\frac{1}{a_{n+1}}\left(\sum _{i=n-r+1}^{n+1} 
a_i\right) \frac{\mult (\mu_{k, n+1})}
{\mult (\sigma_k)}\\&\leq (r+1)\frac{\mult (\mu_{k, n+1})}
{\mult (\sigma_k)}
\end{split}
\end{equation} 
holds for every $n-r+1\leq k\leq n$. 
By definition, we know that 
\begin{equation}
\frac{\mult (\sigma_k)}{\mult (\mu _{k, n+1})}
\end{equation} 
is a positive integer. 
Hence \eqref{equation3.3} implies that 
\begin{equation}
\mult(\mu_{k, n+1})=\mult (\sigma_k)
\end{equation}
holds for every $n-r+1\leq k\leq n$. 
This means $v_{n+1}$ generates $N/N_{\mu_{k,n+1}}$. 
Therefore, 
for every 
$n-r+1\leq k\leq n$, 
the equality $a_kv_k+a_{n+1}v_{n+1}=0$ in $N/N_{\mu_{k,n+1}}$ 
tells us that 
$a_k$ divides $a_{n+1}$. 
By \eqref{equation3.3}, we obtain the following claim. 
Though the proof is completely similar to the proof of the claim in 
\cite[Proposition 2.9]{fujino}, we describe it for 
the sake of completeness. 
\begin{claim}
\[
a_{n-r+1}=\cdots =a_{n+1}=1. 
\]
\end{claim} 
\begin{proof}[Proof of Claim]
Suppose that $a_{n-r+1}\neq a_{n+1}$. 
Since 
\[
v_{n-r+1}=-\frac{1}{a_{n-r+1}}\sum_{i=n-r+2}^{n+1}a_iv_{n+1}
\]
is a primitive vector, $a_{n-r+2}\neq a_{n+1}$ also holds. Namely, 
\[
\frac{a_{n-r+1}}{a_{n+1}},\ \frac{a_{n-r+2}}{a_{n+1}}\leq\frac{1}{2},
\]
and this contradicts \eqref{equation3.3}. 
\end{proof}
Thus, \eqref{equation3.2} is nothing but 
\begin{equation}
v_{n-r+1}+\cdots +v_{n+1}=0. 
\end{equation}
Since this equality says that $v_i=-v_{n+1}$ in $N/N_{\mu_{i,n+1}}$ 
for every $n-r+1\le i\le n$, 
$v_i$ generates $N/N_{\mu_{i,n+1}}$, 
that is, 
we have an isomorphism 
\begin{equation}\label{equation3.4}
\mathbb Zv_i\overset{\sim}{\longrightarrow} N/N_{\mu_{i, n+1}}.
\end{equation} 
Let $v$ be any element of $N$. 
Then, by \eqref{equation3.4}, we can find 
$b_{n-r+1}, \ldots, b_n\in \mathbb Z$ such that 
\begin{equation}
v-(b_{n-r+1}v_{n-r+1}+\cdots +b_nv_n)\in N_{\langle 
v_1, \ldots, v_{n-r}\rangle}. 
\end{equation} 
This implies that 
$\{v_{n-r+1}, \ldots, v_{n}\}$ spans 
$N_{\langle v_{n-r+1}, \ldots, v_n\rangle}$ and 
that there exists 
a splitting $N=N_{\langle v_{n-r+1}, \ldots, v_n\rangle}
\oplus N_{\langle v_1, \ldots, v_{n-r}\rangle}$. 
The natural projection map 
\begin{equation*}
N\to N/N_{\langle v_{n-r+1}, \ldots, v_n\rangle}
\end{equation*} 
and the fan $\Sigma$ define a fan $\Sigma_Y$ in 
$N/N_{\langle v_{n-r+1}, \ldots, v_n\rangle}$. 
Then we obtain a toric extremal contraction morphism 
of fibering type 
\begin{equation*}
\varphi_R\colon X=X(\Sigma)\to Y:=Y(\Sigma_Y). 
\end{equation*} 
For the details of the above description of 
toric extremal contractions, 
see e.g.~\cite[Corollary 14-2-2]{matsuki}. 
Since $\{v_{n-r+1}, \ldots, v_{n}\}$ spans 
$N_{\langle v_{n-r+1}, \ldots, v_n\rangle}$, 
\begin{equation*}
v_{n-r+1}+\cdots +v_{n+1}=0, 
\end{equation*}  
and there exists 
a splitting 
\begin{equation*}
N=N_{\langle v_{n-r+1}, \ldots, v_n\rangle}
\oplus N_{\langle v_1, \ldots, v_{n-r}\rangle}, 
\end{equation*}  
the extremal contraction 
$\varphi_R\colon X\to Y$ is a $\mathbb P^r$-bundle 
(see e.g.~\cite[Exercise.~(Fiber bundles) on page 41]{fulton}). 
Hence, we can easily check that $\mathscr F=\mathscr T_{X/Y}$ 
(see e.g.~\cite[Proposition 3.1.6]{pang}) 
and 
that $l_{\mathscr F}(R)=r+1$ holds under the assumption that 
$l_{\mathscr F}(R)>r$. 
Thus we obtain all the desired properties. 
We finish the proof. 
\end{proof}


\end{document}